\documentclass[10pt,reqno]{amsart}
\usepackage{amsmath}
\usepackage{amssymb}
\usepackage{amsthm}

\newlength{\defbaselineskip}
\newcommand{\setlinespacing}[1]%
           {\setlength{\baselineskip}{#1 \defbaselineskip}}

\numberwithin{equation}{section}

\newtheorem{thm}{Theorem}[section]

\newtheorem{lem}[thm]{Lemma}

\theoremstyle{definition}

\theoremstyle{remark}

\numberwithin{equation}{section}

\begin{document}
\title[The radius of spatial analyticity]
{Lower bounds on the radius of spatial analyticity for the Kawahara equation}

\author{Jaeseop Ahn, Jimyeong Kim and Ihyeok Seo}

\thanks{This research was supported by NRF-2019R1F1A1061316.}


\address{Department of Mathematics, Sungkyunkwan University, Suwon 16419, Republic of Korea}
\email{j.ahn@skku.edu}

\email{jimkim@skku.edu}

\email{ihseo@skku.edu}

\begin{abstract}
In this paper we obtain lower bounds on the radius of spatial analyticity of solutions to the Kawahara equation
$u_t  + uu_x + \alpha u_{xxx} + \beta u_{xxxxx} = 0$, $\beta\neq0$, given initial data which is analytic with a fixed radius.
It is shown that the uniform radius of spatial analyticity of solutions at later time $t$ can decay no faster than $1/|t|$ as $|t|\rightarrow\infty$.
\end{abstract}

\maketitle

\section{Introduction}\label{sec1}

Consider the Cauchy problem for the Kawahara equation
\begin{equation}\label{Kawahara}
\begin{cases}
u_t  + uu_x + \alpha u_{xxx} + \beta u_{xxxxx} = 0,\\
u(0, x) = u_0(x),
\end{cases}
\end{equation}
where $u:\mathbb{R}^{1+1}\rightarrow\mathbb{R}$ and $\alpha$, $\beta$ are real constants with $\beta \neq 0$.
This fifth-order KdV type equation has been derived to model gravity-capillary waves on a shallow layer
and magneto-sound propagation in plasmas (\cite{Ka,KS}).

The well-posedness of the above Cauchy problem with initial data in Sobolev spaces $H^s$ has been studied by several authors
(see e.g. \cite{CDT,WCD,CLMW,CG}).
In particular, it was shown in \cite{CDT} that \eqref{Kawahara} has a global solution when $s=0$.
This was improved by Wang, Cui and Deng \cite{WCD} to $s>-1/2$
and then by Chen, Li, Miao and Wu \cite{CLMW} to $s>-7/4$.
More recently, Chen and Guo \cite{CG} obtained the global well-posedness for $s\geq-7/4$.

In this paper, we are concerned with the persistence of spatial analyticity for the solutions of \eqref{Kawahara},
given initial data in a class of analyticity functions.
While the well-posedness theory in Sobolev spaces is well developed, nothing is known about the spatial analyticity for the Kawahara equation.
From now on, we focus on the situation where we consider a real-analytic initial data
with uniform radius of analyticity $\sigma_0>0$,
so there is a holomorphic extension to a complex strip
$$S_{\sigma_0}=\{x+iy:x,y\in\mathbb{R},\,|y|<\sigma_0\}.$$
Now, it is natural to ask whether this property may be continued analytically to a complex strip $S_{\sigma(t)}$ for all later times $t$,
but with a possibly smaller and shrinking radius of analyticity $\sigma(t)>0$.

This type of problem was first introduced by Kato and Masuda \cite{KM},
and has recently received a lot of attention for the Korteweg-de Vries (KdV) equation.
It was shown by Bona and Gruji\'{c} \cite{BG} that
the radius $\sigma(t)$ for the KdV equation can decay no faster than $e^{-t^2}$ as $|t|\rightarrow\infty$.
This was improved greatly by Bona, Gruji\'{c} and Kalisch \cite{BGK} to a polynomial decay rate of $|t|^{-12}$.
Later, Selberg and da Silva \cite{SS} obtained a further refinement, $\sigma(t)\geq c|t|^{-4/3+\varepsilon}$,
where the $\varepsilon$ exponent was also removed by Tesfahun \cite{Te2}. The rate was again improved by Huang and Wang \cite{HW} up to $|t|^{-1/4}$.
See also a recent related result for the quartic generalised KdV equation by Selberg and Tesfahun \cite{ST2}.

In spite of these many works for KdV equations, there have been no results on this issue for the Kawahara equation \eqref{Kawahara}
which is a fifth-order KdV type equation.
Motivated by this, we aim here to obtain the spatial analyticity for the Kawahara equation.

The Gevrey space, denoted $G^{\sigma,s}(\mathbb{R})$, $\sigma\geq0$, $s\in\mathbb{R}$, is a suitable function space to
study analyticity of solution. In our case, it will be used with the norm
$$\|f\|_{G^{\sigma,s}}=\| e^{\sigma|D|} \langle D\rangle^s f\|_{L^2},$$
where $\langle D\rangle=1+|D|$ and $D$ denotes the derivative.
According to the Paley-Wiener theorem\footnote{The proof given for $s=0$ in \cite{K} applies also for $s\in\mathbb{R}$
with some obvious modifications.} (see e.g. \cite{K}, p. 209),
a function $f$ belongs to $G^{\sigma,s}$ with $\sigma>0$
if and only if it is the restriction to the real line of a function $F$ which is holomorphic in the strip
$S_\sigma=\{ x+iy:x,y\in\mathbb{R},\,|y| < \sigma \}$ and satisfies
$\sup_{|y| < \sigma} \| F(x+iy)\|_{H_x^s} <\infty$.
Therefore every function in $G^{\sigma,s}$ with $\sigma>0$ has an analytic extension to the strip $S_\sigma$.
Based on this property of the Gevrey space, which is the key to studying spatial analyticity of solution,
our result below gives a lower bound $|t|^{-1}$ on the radius of analyticity $\sigma(t)$ of the solution to \eqref{Kawahara}
as the time $t$ tends to infinity.

\begin{thm}\label{thm1}
Let $u$ be the global $C^\infty$ solution of \eqref{Kawahara} with $u_0 \in G^{\sigma_0,s}(\mathbb{R})$ for some $\sigma_0>0$ and $s\in\mathbb{R}$.
Then, for all $t\in\mathbb{R}$
$$u(t)\in G^{\sigma(t),s}(\mathbb{R})$$
with $\sigma(t)\geq c|t|^{-1}$ as $|t| \rightarrow \infty$.
Here, $c>0$ is a constant depending on $\|u_0\|_{G^{\sigma_0,s}(\mathbb{R})}$.
\end{thm}

It should be noted that the existence of the global $C^\infty$ solution in the theorem is always guaranteed.
Indeed, observe that ${G^{0,s}}$ coincides with the Sobolev space $H^s$ and the embeddings
\begin{equation}\label{emb}
G^{\sigma,s}\subset G^{\sigma^\prime,s^\prime}
\end{equation}
hold for all $0\leq\sigma'<\sigma$ and $s,s'\in\mathbb{R}$.
As a consequence of this embedding with $\sigma'=0$ and the existing global well-posedness theory in $H^{s'}(=G^{0,s'})$ for $s'\geq-7/4$,
the Cauchy problem \eqref{Kawahara} has a unique smooth solution for all time, given initial data $u_0\in G^{\sigma_0,s}$
for any $\sigma_0>0$ and $s\in\mathbb{R}$.

We close this section with further references on the spatial analyticity for other dispersive equations such as
Schr\"odinger equations \cite{BGK2,Te,AKS}, Klein-Gordon equations \cite{P} and Dirac-Klein-Gordon equations \cite{ST,S}.

The outline of this paper is as follows:
In Section \ref{sec2} we introduce some function spaces such as Bourgain and Gevrey-Bourgain spaces, and their basic properties
which will be used in later sections.
In Section \ref{sec3} we present a bilinear estimate (Lemma \ref{bilinear}) in Gevrey-Bourgain spaces.
By making use of a contraction argument involving this estimate,
we prove that in a short time interval $0\leq t \leq\delta$ with $\delta>0$ depending on the norm of the initial data,
the radius of analyticity remains strictly positive.
Next, we prove an approximate conservation law, although the conservation of $G^{\sigma_0,s}$-norm of the solution does not hold exactly,
in order to control the growth of the solution in the time interval $[0, \delta]$, measured in the data norm $G^{\sigma_0,s}$.
Section \ref{sec4} is concerned with the proofs of such a local result and the almost conservation law.
In Section \ref{sec5}, we finish the proof of Theorem \ref{thm1} by iterating the local result based on the conservation law.
The final section, Section \ref{sec6}, is devoted to the proof of Lemma \ref{bilinear}.

Throughout this paper, the letter $C$ stands for a positive constant which may be different
at each occurrence. We denote $A\lesssim B$ and $A\sim B$ to mean $A\leq CB$ and $B\lesssim A\lesssim B$, respectively.
We also use $A\ll B$ to mean $A\leq cB$ for some small constant $c>0$.

\section{Function spaces}\label{sec2}
In this section we introduce some function spaces and their basic properties
which will be used in later sections for the proof of Theorem \ref{thm1}.

For $s,b\in\mathbb{R}$, we use  $X^{s,b}=X^{s,b}(\mathbb{R}^2)$ to denote the Bourgain space defined by the norm
$$
\|f\|_{X^{s,b}} =\| \langle\xi\rangle^s\langle\tau-p(\xi)\rangle^b\widehat{f}(\tau,\xi)\|_{L^2_{\tau,\xi}},
$$
where $p(\xi)=\alpha \xi^3 -\beta\xi^5$, $\langle\cdot\rangle=1+|\cdot|$ and $\widehat{f}$ denotes the space-time Fourier transform given by
$$
\widehat{f}(\tau,\xi)=\int_{\mathbb{R}^{2}} e^{-i(t\tau+x\xi)}f(t,x) \ dtdx.
$$
The restriction of the Bourgain space, denoted $X^{s,b}_\delta$, to a time slab $(0,\delta )\times\mathbb{R}$
is a Banach space when equipped with the norm
$$
\|f\|_{X^{s,b}_\delta}=\inf\big\{\|g\|_{X^{s,b}} : g=f\,\, \text{on}\,\, (0,\delta)\times\mathbb{R}\big\}.
$$
We also introduce the Gevrey-Bourgain space $X^{\sigma,s,b}=X^{\sigma,s,b}(\mathbb{R}^2)$ defined by the norm
$$
\|f\|_{X^{\sigma,s,b}}=\| e^{\sigma| \partial_x|}f\|_{X^{s,b}}.
$$
Its restriction $X^{\sigma,s,b}_\delta$ to a time slab $(0,\delta)\times\mathbb{R}$ is defined in a similar way as above,
and when $\sigma=0$ it coincides with the Bourgain space $X^{s,b}$.

The Gevrey-modification of the Bourgain spaces was used already by Bourgain \cite{B}
to study persistence of analyticity of solutions of the Kadomtsev-Petviashvili equation.
He proved that the radius of analyticity remains positive as long as the solution exists.
His argument is quite general and applies also to the KdV and Kawahara equations,
but it does not give any lower bound on the radius $\sigma(t)$ as $|t|\rightarrow\infty$.

We now list some basic properties of those spaces.
When $\sigma=0$, the proofs of the first two lemmas below can be found in Section 2.6 of \cite{T},
and the third lemma follows by the argument used for Lemma 3.1 of \cite{CKSTT2}.
But, by the substitution $f\rightarrow e^{\sigma|\partial_x|}f$, the properties of $X^{s,b}$ and its restrictions
carry over to $X^{\sigma,s,b}$.

\begin{lem}\label{lem00}
Let $\sigma \geq 0$, $s\in\mathbb{R}$ and $b>1/2$.
Then, $X^{\sigma,s,b}\subset C(\mathbb{R},G^{\sigma,s})$ and
$$\sup_{t\in\mathbb{R}}\|f(t)\|_{G^{\sigma,s}}\leq C\|f\|_{X^{\sigma,s,b}},$$
where $C>0$ is a constant depending only on $b$.
\end{lem}

\begin{lem}\label{lem2}
Let $\sigma \geq 0$, $s\in\mathbb{R}$, $-1/2<b<b^\prime<1/2$ and $\delta>0$. Then
$$
\|f\|_{X^{\sigma,s,b}_\delta} \leq C_{b,b'}\delta^{b^\prime-b}\|f\|_{X^{\sigma,s,b^\prime}_\delta},
$$
where $C_{b,b'}>0$ is a constant depending only on $b$ and $b'$.
\end{lem}

\begin{lem}\label{lem3}
Let $\sigma \geq 0$, $s\in\mathbb{R}$, $-1/2<b<1/2$ and $\delta>0$.
Then, for any time interval $I\subset[0,\delta]$,
$$
\|\chi_I f\|_{X^{\sigma,s,b}} \leq C\|f\|_{X^{\sigma,s,b}_\delta},
$$
where $\chi_I(t)$ is the characteristic function of $I$, and $C>0$ is a constant depending only on $b$.
\end{lem}

\section{Bilinear estimates in Gevrey-Bourgain spaces}\label{sec3}

In this section we present a bilinear estimate in Gevrey-Bourgain spaces, Lemma \ref{bilinear},
which plays a key role in obtaining the local well-posedness and almost conservation law in the next section.
With the aid of it, we shall also deduce an estimate, Lemma \ref{f}, which is another useful tool particularly
in obtaining the almost conservation law.

\begin{lem}\label{bilinear}
For all $\sigma\geq0$ and $s>-7/4$, there exist $1/2<b<1$ and $\varepsilon>0$ such that
\begin{equation}\label{comm}
\| \partial_x (uv)\|_{X^{\sigma,s,b'-1}}\leq C_{s,b,b'}\| u\|_{X^{\sigma,s,b}}\| v\|_{X^{\sigma,s,b}}
\end{equation}
for any $b'$ satisfying $b\leq b'<b+\varepsilon$.
Here, $C_{s,b,b'}>0$ is a constant depending only on $s$, $b$ and $b'$.
\end{lem}

It is worth comparing this lemma with the analogous result by Chen, Li, Miao and Wu \cite{CLMW} (\textit{cf}. Proposition 2.2),
where the proof for \eqref{comm} is given only in the case $b'=b$ to obtain local well-posedness in Sobolev spaces.
On the contrary, it is crucial for the present issue to have $b'$ that can range over a small interval for which \eqref{comm} holds.
Similarly as in \cite{CLMW}, we apply Tao's $[k;Z]$-multiplier norm method \cite{T2}
to our case where $b\leq b'<b+\varepsilon$.

We shall postpone the detailed proof of Lemma \ref{bilinear} until the last section, Section \ref{sec6}.
Instead here we derive the following lemma from Lemma \ref{bilinear},
which, along with the function $f$ defined here, plays a crucial role in obtaining the almost conservation law in the next section.

\begin{lem}\label{f}
Let
\begin{equation}\label{fff}
f(u)=\frac{1}{2}\partial_x\Big((e^{\sigma|\partial_x|}u)^2-e^{\sigma|\partial_x|}u^2\Big).
\end{equation}
Given $0\leq\rho\leq 1$, there exist $1/2<b<1$ and $C>0$ such that
$$
\| f(u)\|_{X^{0,b-1}} \leq C\sigma^\rho\| u\|^2_{X^{\sigma,0,b}}
$$
for all $\sigma>0$ and $v\in X^{\sigma,0,b}$.
\end{lem}

\begin{proof}
Notice first that
\begin{align*}
\widehat{v_1v_2}(\tau,\xi) &= \widehat{v_1}\ast\widehat{v_2}(\tau,\xi)\\
&= \int_{\mathbb{R}^2} \widehat{v_1}(\tau_1,\xi_1)\widehat{v_2}(\tau-\tau_1,\xi-\xi_1)\,d\tau_1d\xi_1\\
&=\int_{\mathbb{R}^2} \widehat{v_1}(\tau_1,\xi_1)\widehat{v_2}(\tau_2,\xi_2) \,d\tau_1d\xi_1
\end{align*}
with $\tau_2=\tau-\tau_1$ and $\xi_2=\xi-\xi_1$.
With this as well as the estimate\footnote{This estimate can be found in Lemma 12 of \cite{SS}.}
$$
e^{\sigma|\xi_1|}e^{\sigma|\xi_2|}-e^{\sigma|\xi_1+\xi_2|}\leq\big(2\sigma \min\{|\xi_1|,|\xi_2|\}\big)^\rho e^{\sigma|\xi_1|}e^{\sigma|\xi_2|}
$$
where $\sigma>0$, $0\leq\rho\leq 1$ and $\xi_1,\xi_2 \in \mathbb{R}$,
one can see that
\begin{align*}
&\| f(u)\|_{X^{0,b-1}}\\
&\sim\bigg \| \frac{\xi}{\langle\tau-p(\xi)\rangle^{1-b}}\int_{\mathbb{R}^2} e^{\sigma|\xi_1|}\hat{u}(\tau_1,\xi_1)e^{\sigma|\xi_2|}\hat{u}(\tau_2,\xi_2)-e^{\sigma|\xi|}\hat{u}(\tau_1,\xi_1)\hat{u}(\tau_2,\xi_2)\, d\tau_1d\xi_1\bigg \|_{L^2_{\tau,\xi}}\\
&\sim\bigg \| \frac{\xi}{\langle\tau-p(\xi)\rangle^{1-b}}\int_{\mathbb{R}^2}\big(e^{\sigma|\xi_1|}e^{\sigma|\xi_2|}-e^{\sigma|\xi|}\big)\hat{u}(\tau_1,\xi_1)
\hat{u}(\tau_2,\xi_2) \, d\tau_1 d\xi_1\bigg \|_{L^2_{\tau,\xi}}\\
&\lesssim\bigg \| \frac{\xi}{\langle\tau-p(\xi)\rangle^{1-b}}\int_{\mathbb{R}^2}\big[2\sigma \min(|\xi_1|,|\xi_2|)\big]^\rho e^{\sigma|\xi_1|}e^{\sigma|\xi_2|}\hat{u}(\tau_1,\xi_1)\hat{u}(\tau_2,\xi_2) \, d\tau_1 d\xi_1\bigg \|_{L^2_{\tau,\xi}},
\end{align*}
where  $\tau_2=\tau-\tau_1$ and $\xi_2=\xi-\xi_1$.
Here, from the triangle inequality,
$$
\min(|\xi_1|,|\xi_2|)\leq 2\frac{(1+|\xi_1|)(1+|\xi_2|)}{(1+|\xi_1+\xi_2|)}=2\frac{\langle\xi_1\rangle\langle\xi_2\rangle}{\langle\xi\rangle},
$$
and therefore
\begin{align*}
\| f(u)\|_{X^{0,b-1}}&\lesssim\sigma^\rho\bigg \| \frac{\xi\langle\xi\rangle^{-\rho}}{\langle\tau-p(\xi)\rangle^{1-b}}
\int_{\mathbb{R}^2}e^{\sigma|\xi_1|}\langle\xi_1\rangle^{\rho}\hat{u}(\tau_1,\xi_1)
e^{\sigma|\xi_2|}\langle\xi_2\rangle^{\rho}\hat{u}(\tau_2,\xi_2)\, d\tau_1 d\xi_1\bigg \|_{L^2_{\tau,\xi}}\\
&=\sigma^\rho\Big\|\partial_x\Big(e^{\sigma|\partial_x|}\langle\partial_x\rangle^\rho u\cdot
e^{\sigma|\partial_x|}\langle\partial_x\rangle^\rho u\Big)\Big\|_{X^{0,-\rho,b-1}}\\
&\lesssim\sigma^\rho\big\| e^{\sigma|\partial_x|}\langle\partial_x\rangle^{\rho} u\big\|^2_{X^{0,-\rho,b}}\\
&=\sigma^\rho\| u\|^2_{X^{\sigma,0,b}}
\end{align*}
as desired. Here we used Lemma \ref{bilinear} with $\sigma=0$, $s=-\rho$ and $b'=b$ for the second inequality.
\end{proof}

\section{Local well-posedness and almost conservation law}\label{sec4}

In this section we first establish the local well-posedness and then the almost conservation law,
by making use of the bilinear estimate in the previous section.
They lie at the core of the proof of Theorem \ref{thm1} in the next section.

\subsection{Local well-posedness}\label{subsec4.1}

Based on Picard's iteration in the $X_\delta^{\sigma,s,b}$-space and Lemma \ref{lem00},
we establish the following local well-posedness in $G^{\sigma, s}$, with a lifespan $\delta > 0$.
Equally the radius of analyticity remains strictly positive in a short time interval $0 \leq t \leq\delta$,
where $\delta > 0$ depends on the norm of the initial data.

\begin{thm}\label{thm2}
Let $\sigma>0$ and $s>-7/4$.
Then, for any $u_0 \in G^{\sigma,s}$, there exist $\delta > 0$ and a unique solution u of the Cauchy problem \eqref{Kawahara} on the time interval $[0,\delta]$ such that $u \in C([0,\delta],G^{\sigma,s})$ and the solution depends continuously on the data $u_0$.
Here we have
\begin{equation}\label{deldel}
\delta = c_0(1+\| u_0 \|_{G^{\sigma.s}})^{-a}
\end{equation}
for some constants $c_0>0$ and $a>2$ depending only on $s$.
Furthermore, if\ $1/2<b<1$, the solution $u$ satisfies
\begin{equation}\label{lowe}
\| u \|_{X^{\sigma,s,b}_\delta} \leq C\| u_0 \|_{G^{\sigma,s}}
\end{equation}
with a constant $C>0$ depending only on $b$.
\end{thm}

\begin{proof}
Fix $\sigma >0$, $s>-7/4$ and $u_0\in G^{\sigma,s}$.
By Lemma \ref{lem00} we shall employ an iteration argument in the space $X^{\sigma,s,b}_\delta$ instead of $G^{\sigma,s}$.

Consider first the Cauchy problem for the linearised Kawahara equation
$$
\begin{cases}
u_t + \alpha u_{xxx} + \beta u_{xxxxx} = F(t,x),\\
u(0,x)=f(x).
\end{cases}
$$
By Duhamel's principle, the solution can be then written as
\begin{equation}\label{DF}
u(t,x)=e^{itp(-i\partial_x)}f(x)+\int_0^t e^{i(t-t')p(-i\partial_x)}F(t',\cdot)dt',
\end{equation}
where the Fourier multiplier $e^{itp(-i\partial_x)}$ with symbol $e^{itp(\xi)}$ is given by
$$e^{itp(-i\partial_x)}f(x)=\frac1{(2\pi)}\int_{\mathbb{R}}e^{ix\xi}e^{itp(\xi)}\widehat{f}(\xi)d\xi.$$
(Recall from Section \ref{sec2} that $p(\xi)=\alpha \xi^3 -\beta\xi^5$.)
Then the following $X_\delta^{\sigma,s,b}$-energy estimate follows directly from Proposition 2.1 in \cite{CLMW} (see also \cite{KPV}).

\begin{lem}\label{lem1}
Let $\sigma \geq 0$, $s\in\mathbb{R}$, $1/2<b\leq 1$ and $0<\delta\leq 1$.
Then we have
$$\| e^{itp(-i\partial_x)}f\|_{X^{\sigma,s,b}_\delta} \leq C_b\|f\|_{G^{\sigma,s}}$$
and
$$\bigg\|\int^t_0 e^{i(t-t')p(-i\partial_x)}F(t',\cdot)dt' \bigg\|_{X^{\sigma,s,b}_\delta} \leq C_b\| F\|_{X^{\sigma,s,b-1}_\delta},$$
with a constant $C_b>0$ depending only on $b$.
\end{lem}

Now, let $\left\{ u^{(n)} \right\}^{\infty}_{n=0}$ be the sequence defined by
$$
\begin{cases}
u^{(0)}_t+\alpha u^{(0)}_{xxx} + \beta u^{(0)}_{xxxxx} = 0,\\
u^{(0)}(0) = u_0(x),
\end{cases}
$$
and for $n\in\mathbb{Z}^+$
$$
\begin{cases}
u^{(n)}_t + \alpha u^{(n)}_{xxx} +\beta u^{(n)}_{xxxxx} = -\frac{1}{2}\partial_x\big(u^{(n-1)}u^{(n-1)}\big),\\
u^{(n)}(0) = u_0(x).
\end{cases}
$$
Applying \eqref{DF}, we first write
$$u^{(0)}(t,x) = e^{itp(-i\partial_x)}u_0(x)$$
and
$$u^{(n)}(t,x) =e^{itp(-i\partial_x)}u_0(x) - \frac{1}{2} \int_0^t e^{i(t-t')p(-i\partial_x)}\partial_x\big(u^{(n-1)}(t',\cdot)u^{(n-1)}(t',\cdot)\big) dt'.$$
By Lemma \ref{lem1} we have
\begin{equation}\label{0step}
\| u^{(0)}\|_{X^{\sigma,s,b}_\delta}\leq C_b\| u_0 \|_{G^{\sigma,s}},
\end{equation}
and Lemmas \ref{lem1}, \ref{lem2} and \ref{bilinear} combined imply
\begin{align}\label{nstep}
\nonumber\| u^{(n)} \|_{X^{\sigma,s,b}_\delta} &\leq C_b\| u_0\|_{G^{\sigma,s}} +
C_b\big\| \partial_x\big(u^{(n-1)}u^{(n-1)}\big)\big\|_{X^{\sigma,s,b-1}_\delta}\\
\nonumber&\leq C_b\| u_0\|_{G^{\sigma,s}} +C_bC_{b,b'}\delta^{b'-b}\big\| \partial_x\big(u^{(n-1)}u^{(n-1)}\big)\big\|_{X^{\sigma,s,b'-1}_\delta}\\
&\leq C_b\| u_0\|_{G^{\sigma,s}} +C_bC_{b,b'}C_{s,b,b'}\delta^{b'-b}\| u^{(n-1)} \|^{2}_{X^{\sigma,s,b}_{\delta}}
\end{align}
with $1/2<b<b^\prime<1$.
By induction together with \eqref{0step} and \eqref{nstep}, it follows that for all $n\geq0$
\begin{equation}\label{proof}
\| u^{(n)}\|_{X^{\sigma,s,b}_\delta} \leq 2C_b\| u_0\|_{G^{\sigma,s}},
\end{equation}
if we choose $\delta$ sufficiently small so that
\begin{equation}\label{delttt}
\delta^{b'-b}\|u_0\|_{G^{s,b}}\leq \frac1{8C_{b,b'}C_{s,b,b'}C_b^2}.
\end{equation}
Using Lemmas \ref{lem1}, \ref{lem2} and \ref{bilinear} together with \eqref{proof} and \eqref{delttt} in that order, we therefore get
\begin{align*}
\| u^{(n)}-&u^{(n-1)}\|_{X^{\sigma,s,b}_\delta}\\
&\leq C_b\big\| \partial_x\big(u^{(n-1)}u^{(n-1)}-u^{(n-2)}u^{(n-2)}\big)\big\|_{X^{\sigma,s,b-1}_\delta} \\
&\leq C_bC_{b,b'}\delta^{b'-b}\big\| \partial_x\big(u^{(n-1)}u^{(n-1)}-u^{(n-2)}u^{(n-2)}\big) \big\|_{X^{\sigma,s,b'-1}_\delta} \\
&\leq C_bC_{b,b'}C_{s,b,b'}\delta^{b'-b}\big\| u^{(n-1)} + u^{(n-2)} \big\|_{X^{\sigma,s,b}_\delta} \big\| u^{(n-1)} - u^{(n-2)} \big\|_{X^{\sigma,s,b}_\delta} \\
&\leq 4C_b^2C_{b,b'}C_{s,b,b'}\delta^{b'-b}\|u_0\|_{G^{s,b}}\big\| u^{(n-1)} - u^{(n-2)} \big\|_{X^{\sigma,s,b}_\delta}\\
&\leq \frac{1}{2}\big\| u^{(n-1)} - u^{(n-2)} \big\|_{X^{\sigma,s,b}_\delta},
\end{align*}
which guarantees the convergence of the sequence $\left\{ u^{(n)} \right\}^{\infty}_{n=0}$ to a solution $u$
with the bound \eqref{proof}.
Furthermore, \eqref{deldel} follows easily from \eqref{delttt} and $0<b'-b<1/2$.

Now assume that $u$ and $v$ are solutions to the Cauchy problem $\eqref{Kawahara}$ for initial data $u_0$ and $v_0$, respectively.
Then similarly as above, again with the same choice of $\delta$ and for any $\delta'$ such that $0<\delta^\prime<\delta$, we have
$$ \| u-v \|_{X^{\sigma,s,b}_{\delta^\prime}} \leq C_b\| u_0 -v_0 \|_{G^{\sigma,s}} + \frac{1}{2}\| u-v \|_{X^{\sigma,s,b}_{\delta^\prime}}$$
provided $\| u_0-v_0\|_{G^{\sigma,s}}$ is sufficiently small, which proves the continuous dependence of the solution on the initial data.

Finally, it remains to show the uniqueness of solutions.
Assume $u, v \in C_tG^{\sigma,s}$ are solutions to $\eqref{Kawahara}$ for the same initial data $u_0$ and let $w=u-v$. Then $w$ satisfies $w_t + \alpha w_{xxx} + \beta w_{xxxxx} + wu_x + vw_x =0$.
Multiplying both sides by $w$ and integrating in space yields
$$
\frac12\int_\mathbb{R} (w^2)_t dx +\alpha\int_\mathbb{R} ww_{xxx} dx +\beta\int_\mathbb{R} ww_{xxxxx} dx
+ \int_\mathbb{R} w^2u_x dx +\int_\mathbb{R} wvw_x dx= 0.
$$
Using $2wvw_x=(vw^2)_x-v_xw^2$ and integrating by parts, we then have
\begin{align*}
\frac{1}{2}\int_\mathbb{R} (w^2)_t dx -\frac{\alpha}{2}\int_\mathbb{R} (w_xw_x)_x dx
&+\frac{\beta}{2}\int_\mathbb{R} (w_{xx}w_{xx})_{x} dx+ \int_\mathbb{R} w^2u_x dx\\
&+\frac12\int_\mathbb{R} (vw^2)_x dx-\frac12\int_\mathbb{R} v_xw^2 dx=0.
\end{align*}
We may here assume that $w$ and its all spatial derivatives decay to zero as $|x|\rightarrow \infty$.\footnote{This property can be shown by approximation using the monotone convergence theorem
and the Riemann-Lebesgue lemma whenever $u\in X_\delta^{\sigma,1,b}$. See the argument in \cite{SS}, p. 1018.}
It follows then that
$$\frac{1}{2}\int_\mathbb{R} (w^2)_t dx =- \int_\mathbb{R} w^2u_x dx
+\frac12\int_\mathbb{R} v_xw^2 dx.$$
By H\"older's inequality, this implies
\begin{align*}
{\frac{d}{dt}}\| w(t) \|^2_{L^2_x} &\leq 2\big(\| u_x(t)\|_{L^\infty_x}  +\|v_x(t)\|_{L^{\infty}_x}\big)\| w(t) \|^2_{L^2_x}\\
&\leq C\| w(t) \|^2_{L^2_x}.
\end{align*}
Here we used the fact that
$$
G^{\sigma,s} \subseteq G^{0,2}=H^2 \subseteq L^q
$$
for all $2\leq q\leq\infty$ and $\sigma>0$.
By Gr\"onwall's inequality, we now conclude that $w=0$.
\end{proof}

\subsection{Almost conservation law}\label{subsec4.2}
We have established the existence of local solutions; we would like to apply the local result repeatedly to cover time intervals of arbitrary length. This, of course, requires some sort of control on the growth of the norm on which the local existence time depends.
The following approximate conservation will allow us (see Section \ref{sec5}) to repeat the local result on successive short-time intervals to reach any target time $T>0$, by adjusting the strip width parameter $\sigma$ according to the size of $T$.

\begin{thm}\label{thm3}
Let $0\leq\rho\leq 1$, $\frac{1}{2}<b<1$ and $\delta$ be as in Theorem \ref{thm2}. Then there exists $C>0$ such that for any $\sigma > 0$ and any solution $u \in X^{\sigma,0,b}_{\delta}$ to the Cauchy problem \eqref{Kawahara} on the time interval $[0,\delta]$, we have the estimate
\begin{equation}\label{acl00}
\sup_{t\in[0,\delta]}\| u(t)\|^2_{G^{\sigma,0}} \leq \| u(0)\|^2_{G^{\sigma,0}} + C\sigma^\rho\| u \|^3_{X^{\sigma,0,b}_\delta}.
\end{equation}
\end{thm}

\begin{proof}
Let $0\leq \delta ' \leq \delta$. Setting $v(t,x)=e^{\sigma|\partial_x|}u(t,x)$ and applying $e^{\sigma|\partial_x|}$ to \eqref{Kawahara}, we obtain
$$
v_t+\alpha v_{xxx}+\beta v_{xxxxx} + vv_x = f(u),
$$
where $f(u)$ is as in \eqref{fff}:
$$
f(u)=\frac{1}{2}\partial_x\Big((e^{\sigma|\partial_x|}u)^2-e^{\sigma|\partial_x|}u^2\Big).
$$
Multiplying both sides by $v$ and integrating in space yield
$$
\int_\mathbb{R} vv_t dx +\alpha\int_\mathbb{R} vv_{xxx} dx +\beta\int_\mathbb{R} vv_{xxxxx} dx + \int_\mathbb{R} v^2v_x dx = \int_\mathbb{R} vf(u) dx.
$$
As before, we may here assume that $v$ and its all spatial derivatives decay to zero as $|x|\rightarrow \infty$.
Using this fact and integration by parts, we have
$$
\frac{1}{2}\int_\mathbb{R} (v^2)_t dx -\frac{\alpha}{2}\int_\mathbb{R} (v_xv_x)_x dx +\frac{\beta}{2}\int_\mathbb{R} (v_{xx}v_{xx})_{x} dx+ \frac{1}{3}\int_\mathbb{R} (v^3)_x dx = \int_\mathbb{R} vf(u) dx,
$$
and furthermore, the second, third and fourth terms on the left side vanish.
Subsequently integrating in time over the interval $[0,\delta ']$, we obtain
$$
\| u(\delta ')\|^2_{G^{\sigma,0}}\leq  \| u(0)\|^2_{G^{\sigma,0}} + 2 \bigg | \int_{\mathbb{R}^2} \chi_{[0,\delta ']}(t)vf(u) dtdx \bigg |.
$$
Now by H\"older's inequailty, Lemma \ref{lem3} and Lemma \ref{f},
\begin{align*}
\bigg | \int_{\mathbb{R}^2} \chi_{[0,\delta ']}(t)vf(u) dtdx \bigg |
&\leq \| \chi_{[0,\delta ']}(t)v\|_{X^{0,1-b}}\| \chi_{[0,\delta ']}(t)f(u)\|_{X^{0,b-1}}\\
&\leq C\|v\|_{X^{0,1-b}_{\delta '}}\|f(u)\|_{X^{0,b-1}_{\delta '}}\\
&\leq C\|u\|_{X^{\sigma,0,1-b}_{\delta '}}\sigma^\rho\| u\|^2_{X^{\sigma,0,b}_{\delta '}}.
\end{align*}
Since $1-b<b$, we therefore get
$$
\sup_{t\in[0,\delta]}\| u(t)\|^2_{G^{\sigma,0}} \leq \| u(0)\|^2_{G^{\sigma,0}} + C\sigma^\rho\| u \|^3_{X^{\sigma,0,b}_\delta}
$$
as desired.
\end{proof}

\section{Proof of Theorem \ref{thm1}}\label{sec5}
By invariance of the Kawahara equation under the reflection $(t.x)\rightarrow(-t,-x)$, we may restrict to positive times. By the embedding \eqref{emb}, the general case $s\in\mathbb{R}$ will reduce to $s=0$ as shown in the end of this section.

\subsection{The case $s=0$}
Combining \eqref{lowe} and \eqref{acl00}, we first note that
\begin{equation}\label{acl01}
\sup_{t\in[0,\delta]}\| u(t)\|^2_{G^{\sigma,0}} \leq \| u(0)\|^2_{G^{\sigma,0}} + C\sigma^\rho\| u(0) \|^3_{G^{\sigma,0}}.
\end{equation}
Let $u_0=u(0)\in G^{\sigma_0,0}$ for some $\sigma_0>0$ and $\delta$ be as in Theorem \ref{thm2}. For arbitrarily large $T$, we want to show that the soution $u$ to \eqref{Kawahara} satisfies
$$
u(t)\in G^{\sigma(t),0} \quad\textrm{for all }\, t\in [0,T],
$$
where
\begin{equation}\label{sigma}
\sigma(t)\geq\frac{c}{T}
\end{equation}
with a constant $c>0$ depending on $\| u_0\|_{G^{\sigma_0,0}}$ and $\sigma_0$.

Now fix $T$ arbitrarily large. It suffices to show
\begin{equation}\label{label}
\sup_{t\in [0,T]}\| u(t)\|^2_{G^{\sigma,0}} \leq 2\| u_0\|^2_{G^{\sigma_0,0}}
\end{equation}
for $\sigma$ satisfying \eqref{sigma}, which in turn implies $u(t)\in G^{\sigma(t),0}$ as desired.

To prove \eqref{label}, we first choose $n\in\mathbb{Z}^+$ so that $n\delta\leq T\leq(n+1)\delta$. Using induction we shall show for any $k\in\left\{1,2,\cdots,n+1\right\}$ that
\begin{equation}\label{label1}
\sup_{t\in[0,k\delta]} \| u(t)\|^2_{G^{\sigma,0}} \leq \| u_0\|^2_{G^{\sigma,0}} + kC\sigma^\rho 2^{3/2}\| u_0\|^3_{G^{\sigma_0,0}}
\end{equation}
and
\begin{equation}\label{label2}
\sup_{t\in [0,k\delta]}\| u(t)\|^2_{G^{\sigma,0}} \leq 2\| u_0\|^2_{G^{\sigma_0,0}},
\end{equation}
provided $\sigma$ satisfies
\begin{equation}\label{label3}
\sigma\leq\sigma_0\quad\text{and}\quad\frac{2T}{\delta}C\sigma^\rho 2^{3/2}\| u_0\|_{G^{\sigma_0,0}}\leq 1.
\end{equation}
Indeed, for $k=1$, we have from \eqref{acl01} that
\begin{align*}
\sup_{t\in[0,\delta]}\| u(t)\|^2_{G^{\sigma,0}} &\leq \| u_0\|^2_{G^{\sigma,0}} + C\sigma^\rho\| u_0 \|^3_{G^{\sigma,0}}\\
&\leq 2\|u_0\|^2_{G^{\sigma_0,0}},
\end{align*}
where we used the fact that $\| u_0\|_{G^{\sigma,0}}\leq \| u_0\|_{G^{\sigma_0,0}}$ and $C\sigma^\rho \| u_0\|_{G^{\sigma_0,0}}\leq 1$
which are a direct consequence of \eqref{label3}.
Now assume \eqref{label1} and \eqref{label2} hold for some $k\in\left\{1,2,\cdots,n\right\}$. Applying \eqref{acl01}, \eqref{label2} and \eqref{label1}, we then have
\begin{align*}
\sup_{t\in[k\delta,(k+1)\delta]}\|u(t)\|^2_{G^{\sigma,0}}&\leq \| u(k\delta)\|^2_{G^{\sigma,0}} + C\sigma^\rho\| u(k\delta)\|^3_{G^{\sigma,0}}\\
&\leq \| u(k\delta)\|^2_{G^{\sigma,0}} + C\sigma^\rho 2^{3/2}\| u_0\|^3_{G^{\sigma_0,0}}\\
&\leq \| u_0\|^2_{G^{\sigma,0}} + C\sigma^\rho (k+1) 2^{3/2}\| u_0\|^3_{G^{\sigma_0,0}}.
\end{align*}
Combining this with the induction hypothesis \eqref{label1} for $k$, we get
\begin{equation}\label{label4}
\sup_{t\in[0,(k+1)\delta]}\|u(t)\|^2_{G^{\sigma,0}}\leq\| u_0\|^2_{G^{\sigma,0}} + C\sigma^\rho (k+1) 2^{3/2}\| u_0\|^3_{G^{\sigma_0,0}}
\end{equation}
which proves \eqref{label1} for $k+1$. Since $k+1\leq n+1\leq T/\delta+1\leq2T/\delta$, from \eqref{label3} we also get
$$
C\sigma^\rho (k+1) 2^{3/2}\| u_0\|_{G^{\sigma_0,0}}\leq \frac{2T}{\delta}C\sigma^\rho  2^{3/2}\| u_0\|_{G^{\sigma_0,0}}\leq 1
$$
which, along with \eqref{label4}, proves \eqref{label2} for $k+1$.

Finally, the condition \eqref{label3} is satisfied for
$$
\sigma = \bigg (\frac{\delta}{C2^{5/2}\|u_0\|_{G^{\sigma_0,0}}}\bigg )^{1/\rho}\bigg(\frac{1}{T}\bigg)^{1/\rho}.
$$
Particularly when $\rho=1$, the constant $c$ in \eqref{sigma} may be given as
$$
c=\frac{\delta}{C2^{5/2}\|u_0\|_{G^{\sigma_0,0}}}
$$
which depends only on $\| u_0 \|_{G^{\sigma_0,0}}$.

\subsection{The general case $s\in\mathbb{R}$.}
Recall that \eqref{emb} states
$$
G^{\sigma,s}\subset G^{\sigma^\prime,s^\prime}\ \textrm{ for all }\ \sigma>\sigma'\geq 0\ \textrm{ and }\ s,s'\in\mathbb{R}\textrm{.}
$$
For any $s\in\mathbb{R}$ we use this embedding to get
$$
u_0\in G^{\sigma_0,s}\subset G^{\sigma_0/2,0}.
$$
From the local well-posedness result, there is a $\delta=\delta(\| u_0\|_{G^{\sigma_0/2,0}})$ such that
$$
u(t)\in G^{\sigma_0/2,0}\quad\textrm{for }\ 0\leq t\leq\delta\textrm{.}
$$
Similarly as in the case $s=0$, for $T$ fixed greater than $\delta$, we have $u(t)\in G^{\sigma^\prime,0}$ for $t\in[0,T]$
and $\sigma^\prime\geq c/T$ with $c>0$ depending on $\| u_0 \|_{G^{\sigma_0/2,0}}$ and $\sigma_0$.
Applying the embedding again, we conclude
$$
u(t)\in G^{\sigma,s}\quad\text{for}\,\ t\in [0,T]
$$
where $\sigma =\sigma^\prime/2$.

\section{Proof of Lemma \ref{bilinear}}\label{sec6}

This last section is devoted to the proof of Lemma \ref{bilinear}.

\subsection{Preliminaries}
Before we begin the proof, we shall introduce some notations in Tao's $[k;Z]$-multiplier norm method \cite{T2}
with $k=3$ and $Z=\mathbb{R}^2$. Let $\Gamma_3(\mathbb{R}^2)$ denote the \textit{hyperplane}
$$\Gamma_3(\mathbb{R}^2) :=\{(\rho_1, \rho_2, \rho_3)\in (\mathbb{R}^2)^3 : \rho_1 + \rho_2 + \rho_3 =(0, 0)\}$$
where each $\rho_i=(\tau_i,\xi_i)$ is an ordered pair of real numbers $\tau_i,\xi_i$.
Note here that $\tau_1+\tau_2+\tau_3=\xi_1 + \xi_2 + \xi_3=0$.
We will also denote $\rho$, $\xi$ and $\tau$ the triplets $(\rho_1, \rho_2, \rho_3)$, $(\tau_1, \tau_2, \tau_3)$ and $(\xi_1, \xi_2, \xi_3)$, respectively. We endow $\Gamma_3(\mathbb{R}^2)$ with the obvious measure
$$\int_{\Gamma_3(\mathbb{R}^2)} f(\rho):=\int_{(\mathbb{R}^2)^2} f(\rho_1, \rho_2, -\rho_1-\rho_2)\,d\rho_1 d\rho_2$$
where $d\rho_i=d\tau_i d\xi_i$ for $i=1,2$.

For any function $m : \Gamma_3(\mathbb{R}^2)  \rightarrow \mathbb{C}$, we define the $[3,\mathbb{R}^2]$-multiplier norm $\|m\|_{[3,\mathbb{R}^2]}$ to be the best constant $c$ such that the inequality
$$\bigg|\int_{\Gamma_3(\mathbb{R}^2)} m(\rho)\prod_{j=1}^3 f_j(\rho_j)\bigg|\leq c\prod_{j=1}^3\|f_j\|_{L^2(\mathbb{R}^{2})}$$
holds for all test functions $f_j$ on $\mathbb{R}^{2}$.

Capitalised variables such as $N_j$, $L_j$, $H$ are presumed to be dyadic, i.e., these variables range over numbers of the form $2^k$ for $k\in\mathbb{Z}$. Let $N_1$, $N_2$, $N_3>0$. It will be convenient to define the quantities $N_{max}\geq N_{med}\geq N_{min}$ to be the maximum, median and minimum of $N_1$, $N_2$, $N_3$, respectively.
Similarly we define $L_{max}\geq L_{med}\geq L_{min}$ whenever $L_1$, $L_2$, $L_3>0$. We also adopt the following summation conventions. Any summation of the form $L_{max}\sim\cdots$ is a sum over the three dyadic variables $L_1$, $L_2$, $L_3\gtrsim 1$, thus for instance
$$\sum_{L_{max}\sim H}:=\sum_{L_1, L_2, L_3\gtrsim 1:\ L_{max}\sim H}.$$
Similarly, any summation of the form $N_{max}\sim\cdots$ is a sum over the three dyadic variables $N_1$, $N_2$, $N_3>0$, thus for instance
$$\sum_{N_{max}\sim N_{med}\sim H}:=\sum_{N_1, N_2, N_3:\ N_{max}\sim N_{med}\sim H}.$$

If $\tau_j$ and $\xi_j$ are given for $j=1,2,3$, we define
$$\lambda_j:=\tau_j-p(\xi_j),$$
where $p(\xi_j)=\alpha \xi_j^3 -\beta\xi_j^5$ is as in Section $\ref{sec2}$.
Here we will let
$$h(\xi):=p(\xi_1)+p(\xi_2)+p(\xi_3)=-\lambda_1-\lambda_2-\lambda_3.$$
By dyadic decomposition of $\xi_j$, $\lambda_j$ and $|h(\xi)|$, one is led to consider
\begin{equation}\label{el1}
\|X_{N_1, N_2, N_3; H; L_1, L_2, L_3}\|_{[3;\mathbb{R}^2]}
\end{equation}
where $X_{N_1, N_2, N_3; H; L_1, L_2, L_3}$ is the multiplier
$$X_{N_1, N_2, N_3; H; L_1, L_2, L_3}(\xi, \tau):=\chi_{|h(\xi)|\sim H}\prod_{j=1}^3\chi_{|\xi_j|\sim N_j}\chi_{|\lambda_j|\sim L_j}.$$
Since the following identities
\begin{equation}\label{el2}
\xi_1+\xi_2+\xi_3=0
\end{equation}
and
\begin{equation}\label{el3}
\lambda_1+\lambda_2+\lambda_3+h(\xi)=0
\end{equation}
hold, from the support of the multiplier, one can see that $X_{N_1, N_2, N_3; H; L_1, L_2, L_3}$ vanishes unless
both
\begin{equation}\label{duck}
N_{max}\sim N_{med}\quad \text{and}\quad L_{max}\sim\max\{H, L_{med}\}
\end{equation}
are satisfied.
We then recall the following fundamental estimates on dyadic blocks for the Kawahara equation.

\begin{lem}[\cite{CLMW}]\label{el4}
Let $H,N_1,N_2,N_3,L_1,L_2,L_3>0$ obey \eqref{el2}, \eqref{el3} and $N_{max}\sim N_{med}\gtrsim 1$.
Then
\begin{equation}\label{el5}
H\sim N_{max}^4 N_{min},
\end{equation}
and
\begin{itemize}
\item If $N_{max}\sim N_{min}$ and $L_{max}\sim H$,
\begin{equation}\label{el6}
\eqref{el1}\lesssim L_{min}^{1/2}N_{max}^{-2}L_{med}^{1/2}.
\end{equation}

\item If $N_{\sigma(2)}\sim N_{\sigma(3)}\gg N_{\sigma(1)}$ and $H\sim L_{\sigma(1)}\gtrsim L_{\sigma(2)},\ L_{\sigma(3)}$,
\begin{equation}\label{el7}
\eqref{el1}\lesssim L_{min}^{1/2}N_{max}^{-2}\min\Big\{ H, \frac{N_{max}}{N_{min}}L_{med}\Big\}^{1/2}
\end{equation}
where $\sigma$ is a permutation in $\{1,2,3\}$.

\item In all other cases,
\begin{equation}\label{el8}
\eqref{el1} \lesssim L_{min}^{1/2}N_{max}^{-2}\min\{ H, L_{med}\}^{1/2}.
\end{equation}
\end{itemize}
\end{lem}

\subsection{Proof}
Now we turn to the proof of Lemma \ref{bilinear}.
By the definition of $X^{s,b}$-norms and the dual characterisation of $L^2$ space, we may show that
\begin{equation}\label{app01}
\bigg \| \frac{(\xi_1+\xi_2)\langle\xi_1\rangle^{-s}\langle\xi_2\rangle^{-s}
\langle\xi_3\rangle^s}{\langle\tau_1-p(\xi_1)\rangle^b\langle\tau_2-p(\xi_2)\rangle^b\langle\tau_3-p(\xi_3)\rangle^{1-b'}} \bigg \|_{[3;\mathbb{R}^2]}\lesssim 1.
\end{equation}
To show this, we first decompose dyadically $|\xi_j|,|\lambda_j|,|h(\xi)|$ as $|\xi_j|\sim N_j$, $|\lambda_j|\sim L_j$, $|h(\xi)|\sim H$.
By some properties of the $[k;Z]$-multiplier norm,
one may restrict the multiplier in \eqref{app01} to the region $L_j \gtrsim 1$ $(j=1,2,3)$ and $\max\{N_1, N_2, N_3\}\gtrsim 1$.
Consequently, it suffices to show that
\begin{align*}
\bigg\|\sum_{N_{max}\gtrsim 1} \sum_{H} \sum_{L_1,L_2,L_3\gtrsim1}
\frac{N_3\langle N_3\rangle^s}{\langle N_1\rangle^s\langle N_2\rangle^sL_1^bL_2^bL_3^{1-b'}} \chi_{N_1,N_2,N_3;H;L_1,L_2,L_3}\bigg\|_{[3;\mathbb{R}^2]}\lesssim 1.
\end{align*}
By Lemma 3.11 (Schur's test) in \cite{T2},
this estimate is reduced to showing that
\begin{align}\label{app02}
\nonumber\sum_{N_{max}\sim N_{med} \sim N} \sum_{H} \sum_{L_{max}\sim H}&
\frac{N_3\langle N_3\rangle^s}{\langle N_1\rangle^s\langle N_2\rangle^sL_1^bL_2^bL_3^{1-b'}}\\
&\quad\times\|\chi_{N_1,N_2,N_3;H;L_1,L_2,L_3}\|_{[3;\mathbb{R}^2]}\lesssim 1
\end{align}
and
\begin{align}\label{app03}
\nonumber\sum_{N_{max}\sim N_{med} \sim N}  \sum_{H}\sum_{L_{max} \sim L_{med}\gg H}&
\frac{N_3\langle N_3\rangle^s}{\langle N_1\rangle^s\langle N_2\rangle^sL_1^bL_2^bL_3^{1-b'}}\\
&\quad\times\|\chi_{N_1,N_2,N_3;H;L_1,L_2,L_3}\|_{[3;\mathbb{R}^2]}\lesssim 1
\end{align}
for all $N\gtrsim 1$.

\subsubsection{Proof of \eqref{app03}}
Since $N_{max}\sim N_{med}\sim N\gtrsim 1$, by (\ref{el5}) and (\ref{el8}), it suffices to show
\begin{equation}\label{app04}
\sum_{N_{max}\sim N_{med} \sim N} \sum_{L_{max}\sim L_{med}\gg N^4N_{min}} \frac{N_3\langle N_3\rangle^s}{\langle N_1\rangle^s\langle N_2\rangle^sL_1^bL_2^bL_3^{1-b'}}L^{1/2}_{min}N^{1/2}_{min} \lesssim 1.
\end{equation}
We only need to consider two cases: $N_1 \sim N_2 \sim N$, $N_3 =N_{min}$ and $N_1 \sim N_3\sim N$, $N_2=N_{min}$.
(The other case $N_2 \sim N_3\sim N$, $N_1=N_{min}$ then follows by symmetry.)

In the former case, the estimate \eqref{app04} can be further reduced to
$$
\sum_{N_{min}\lesssim N} \sum_{L_{max}\sim L_{med}\gg N^4N_{min}} \frac{N^{-2s}N_{min}\langle N_{min}\rangle^s}{L^b_{min}L^b_{med}L^{1-b'}_{max}}L^{1/2}_{min}N^{1/2}_{min} \lesssim 1
$$
since $1-b'<1/2<b$.
Then performing the $L$ summations, we reduce to
$$
\sum_{N_{min}\lesssim N}  \frac{N^{-2s}N_{min}^{3/2}\langle N_{min}\rangle^s}{(N^4 N_{min})^{1+b-b'}} \lesssim 1,
$$
which holds if $4(1+b-b')+2s>0$.
So we require $0\leq b'-b<1+s/2$, which is possible if $s>-2$.
Therefore, \eqref{app04} follows.

In the latter case, the estimate \eqref{app04} can be reduced to
$$
\sum_{N_{min}\lesssim N} \sum_{L_{max}\sim L_{med}\gg N^4N_{min}} \frac{N}{\langle N_{min}\rangle^sL^b_{min}L^b_{med}L^{1-b'}_{max}}L^{1/2}_{min}N^{1/2}_{min} \lesssim 1.
$$
Before performing the $L$ summations, we need to divide cases into two parts, $N_{min}\leq 1$ and $N_{min}\geq 1$, as follows:
\begin{align*}
&\sum_{N_{min}\leq 1}\sum_{L_{max}\sim L_{med}\gg N^4N_{min}} \frac{NN^{1/2}_{min}}{L^{b-1/2}_{min}L_{max}^{1-b+b'}}\\
+&\sum_{1\leq N_{min}\lesssim N} \sum_{L_{max}\sim L_{med}\gg N^4N_{min}} \frac{NN^{1/2-s}_{min}}{L^{b-1/2}_{min}L_{max}^{1-b+b'}} \lesssim 1.
\end{align*}
This holds clearly if $s>-7/2$, and therefore \eqref{app04} holds for $s>-7/2$.

\subsubsection{Proof of \eqref{el6}}

The conditions
$N_{max}\sim N_{med} \sim N\gtrsim 1$ and $L_{max}\sim H$ in \eqref{app02} are divided into three cases
corresponding to each of \eqref{el6}, \eqref{el7} and \eqref{el8}.
By (\ref{el5}) we also see that $L_{max}\sim N^4_{max}N_{min}$,
which is used repeatedly below.

\subsubsection*{The case for \eqref{el6}}

Now we shall consider the first case, which reduces to
\begin{equation}\label{app05}
\sum_{L_{max}\sim N^5} \frac{N^{-s}N}{L^b_{min}L^b_{med}L^{1-b'}_{max}}L^{1/2}_{min}N^{-2}L^{1/2}_{med}\lesssim 1.
\end{equation}
Performing the $L$ summations, we reduce to
$$
\frac{1}{N^{1+s}N^{5(1-b')}}\lesssim 1,
$$
which holds if $1+s+5(1-b')\geq0$. So we require $1/2< b'\leq(6+s)/5$, which is possible if $s>-7/2$.
Therefore, \eqref{app05} follows.

\subsubsection*{The case for \eqref{el7}}

Next we consider the second case which is divided into $N_1 \sim N_2 \gg N_3,H \sim L_3 \gtrsim L_1,L_2$ and
$N_2 \sim N_3 \gg N_1, H \sim L_1 \gtrsim L_2,L_3$.
(The other case $N_1 \sim N_3 \gg N_2, H \sim L_2 \gtrsim L_1,L_3$ then follows by symmetry.)
In the former case the estimate \eqref{app02} reduces to
\begin{equation}\label{app06}
\sum_{N_3 \ll N}\sum_{1\lesssim L_1,L_2\lesssim N^4N_3}\frac{N_3\langle N_3\rangle^s}{N^{2s}L^b_1L^b_2L^{1-b'}_3}L^{1/2}_{min}N^{-2}\min\bigg\{N^4N_3,\frac{N}{N_3}L_{med}\bigg\}^{1/2}\lesssim 1.
\end{equation}
We then decompose the left-hand side of \eqref{app06} into two parts $N_3\leq 1$ and $N_3> 1$:
\begin{align*}
&\sum_{N_3\leq 1}\sum_{1\lesssim L_1,L_2\lesssim N^4N_3} \frac{N_3\langle N_3\rangle^s}{N^{2s}L^b_1L^b_2L^{1-b'}_3}L^{1/2}_{min}N^{-2}\min\bigg\{N^4N_3,\frac{N}{N_3}L_{med}\bigg\}^{1/2}\\
&+ \sum_{1<N_3\ll N}\sum_{1\lesssim L_1,L_2\lesssim N^4N_3} \frac{N_3\langle N_3\rangle^s}{N^{2s}L^b_1L^b_2L^{1-b'}_3}L^{1/2}_{min}N^{-2}\min\bigg\{N^4N_3,\frac{N}{N_3}L_{med}\bigg\}^{1/2}\\
&=:I_1 +I_2.
\end{align*}
We first consider the part $I_1$. When $N^4N_3\geq \frac{N}{N_3}L_{med}$, which is equivalent to $N_3\geq( L_{med}/N^3)^{1/2}$,
we see that
$$
I_1\lesssim\sum_{N_3\leq1}\sum_{1\lesssim L_1,L_2\lesssim N^4N_3\atop (L_{med}/N^3)^{1/2}\leq N_3}
\frac{N_3}{N^{2s}L^{b-1/2}_{min}L^b_{med}(N^4N_3)^{1-b'}}N^{-2}N^{1/2}L^{1/2}_{med}N^{-1/2}_3.
$$
Performing the $N_3$ summation, we have the desired bound
$$
I_1 \lesssim\sum_{1\lesssim L_1,L_2\lesssim N^4}\frac{N^{-3/2}}{N^{2s}L^{b-1/2}_{min}L^{b-1/2}_{med}N^{4(1-b')}}
\lesssim 1$$
if $2s+4(1-b')+3/2>0$.
So we require $1/2<b\leq b'<(4s+11)/8$, which is possible if $s>-7/4$.
On the other hand, when $N_3<( L_{med}/N^3)^{1/2}$,
$$
I_1\lesssim\sum_{N_3\leq 1}\sum_{1\lesssim L_1,L_2\lesssim N^4N_3\atop N_3<( L_{med}/N^3)^{1/2}} \frac{N_3}{N^{2s}L^{b-1/2}_{min}L^b_{med}(N^4N_3)^{1-b'}}N^{-2}N^2 N^{1/2}_3.
$$
Performing the $N_3$ summation, we have
$$
I_{1} \lesssim \sum_{1\lesssim L_1,L_2\lesssim N^4}\frac{\min\{1,( L_{med}/N^3)^{\frac{1}{2}(\frac{1}{2}+b')}\}}{N^{2s}L^{b-1/2}_{min}L^{b}_{med}N^{4(1-b')}}
\lesssim 1
$$
if $2s+4(1-b')+\frac{3}{2}(\frac{1}{2}+b')>0$ when $(L_{med}/N^3)^{\frac{1}{2}(\frac{1}{2}+b')}\leq1$,
and if $2s+4+3b-4b'>0$ when $(L_{med}/N^3)^{\frac{1}{2}(\frac{1}{2}+b')}\geq1$.
So we require
$1/2<b\leq b'<(8s+19)/10$ and
 $1/2<b'+3(b'-b)<4+2s$,
 but this is possible if $s>-7/4$.
Now we consider the part $I_2$ that has the following trivial bound
\begin{equation*}
I_2\lesssim \sum_{1<N_3\ll N} \sum_{1\lesssim L_1,L_2\lesssim N^4N_3} \frac{N_3 N^s_3}{N^{2s}L^{b-1/2}_{min}L^b_{med}(N^4N_3)^{1-b'}}N^{-2}N^{1/2}L^{1/2}_{med}N^{-1/2}_3.
\end{equation*}
Performing the $N_3$ summation, we conclude that
$$
I_2\lesssim \sum_{1\lesssim L_1,L_2\lesssim N^5}\frac{\min\{1,N^{s-1/2+b'}\}}{N^{2s}L^{b-1/2}_{min}L^{b-1/2}_{med}N^{4(1-b')}N^{3/2}}\lesssim 1
$$
if $2s+4(1-b')+3/2>0$ when $s-1/2+b'< 0$, and
if $2s+4(1-b')+3/2>s-1/2+b'$ when $s-1/2+b'\geq 0$.
So we require $1/2<b\leq b'<\min\{(4s+11)/8,(s+6)/5\}$.
This is possible if $s>-7/4$.
Consequently, we get the desired estimate \eqref{app06} if $s>-7/4$.

Next we deal with the remaining case where $N_2\sim N_3 \gg N_1, H\sim L_1 \gtrsim L_2,L_3$.  In this case,
applying \eqref{el7} to \eqref{app02}, we may show
\begin{equation*}
\sum_{N_1\ll N}\sum_{1\lesssim L_2,L_3 \lesssim N^4N_1} \frac{N^{1+s}L^{1/2}_{min}}{N^s\langle N_1\rangle^sL^b_2L^{1-b'}_3(N^4N_1)^b}N^{-2}\min\bigg\{ H,\frac{N}{N_1}L_{med}\bigg\}^{1/2}\lesssim 1.
\end{equation*}
We first decompose the left-hand side of this inequality into two parts $N_1\leq 1$ and $N_1> 1$:
\begin{align*}
&\sum_{N_1\leq 1}\sum_{1\lesssim L_2,L_3 \lesssim N^4N_1} \frac{N^{1+s}}{N^sL^b_{min}L^{1-b'}_{med}(N^4N_1)^b}L^{1/2}_{min}N^{1/2}_1 \\
&+\sum_{1<N_1\ll N }\sum_{1\lesssim L_2,L_3 \lesssim N^4N_1} \frac{N^{1+s}}{N^s_1N^sL^b_{min}L^{1-b'}_{med}(N^4N_1)^b}L^{1/2}_{min}N^{-2}N^{5/4}L^{1/4}_{med} \\
&=: J_1+J_2.
\end{align*}
Here we used the fact that $H\sim N^4N_{min}$ and $\min\{a,b\}\leq\sqrt{ab}$.
Note first that the $L$ summations in $J_1$ vanish unless $N^4N_1\gtrsim 1$.
Using this,  we get
$$
J_1\lesssim \sum_{N^{-4}\lesssim N_1\leq 1} \frac{NN^{1/2-b}_1}{N^{4b}}\lesssim \frac{NN^{(-4)(1/2-b)}}{N^{4b}}\lesssim 1
$$
since $b>1/2$ and $N\gtrsim1$.
Performing the $L$ summations in $J_2$, we see that
$$
J_2\lesssim \sum_{1\leq N_1 \ll N} \frac{N^{1/4}N_1^{-s-b}}{N^{4b}}
$$
if we take $1/2< b\leq b'<3/4$. (Otherwise, the sum may diverge.)
When $-s-b>0$, we conclude
$$
J_2\lesssim  \frac{N^{1/4}N^{-s-b}}{N^{4b}}  \lesssim 1
$$
if $4b+s+b-1/4>0$.
So we require $\max\{1/2,(1/4-s)/5\}<b \leq b'<3/4$ and $b<-s$.
This is possible if $-7/2<s<-1/2$.
Furthermore, we can have the range  $1/2<b \leq b'<3/4$
if $-9/4\leq s<-1/2$.
When $-s-b\leq0$,
$$
J_2\lesssim \frac{N^{1/4}}{N^{4b}}\log_2(N)\lesssim 1
$$
which holds when $1/2<b\leq b'< 3/4$.
Consequently, we get the desired estimate.

\subsubsection*{The case for \eqref{el8}}

Lastly, using \eqref{el8}, the desired estimate \eqref{app02} reduces to
$$
\sum_{N_{max}\sim N_{med} \sim N} \sum_{L_{max}\sim N^4N_{min}}
\frac{N_3\langle N_3\rangle^s}{\langle N_1\rangle^s\langle N_2\rangle^sL^b_1L^b_2L^{1-b'}_3}L^{1/2}_{min}N^{-2}\min\{H,L_{med}\}^{1/2}\lesssim 1.
$$
To show this, we need to divide the case into $N_1 \sim N_2 \sim N,$ $N_3 = N_{min}$ and $N_1 \sim N_3 \sim N,$ $N_2 = N_{min}$.
(The other case $N_2 \sim N_3 \sim N,$ $N_1 = N_{min}$ then follows by symmetry.)
In the former case, the above estimate further reduces to
$$
\sum_{ N_3\ll N} \sum_{L_{max}\sim N^4N_3} \frac{N_3\langle N_3\rangle^s}{N^{2s}L^b_{min}L^b_{med}(N^4N_3)^{1-b'}}L^{1/2}_{min}N^{-2}L^{1/2}_{med}\lesssim 1.
$$
Then performing the $L$ summations, we reduce to
$$
\sum_{N_3\ll N} \frac{N_3\langle N_3\rangle^s}{N^{2+2s}N^{4(1-b')}N^{1-b'}_3} \lesssim 1,
$$
which holds if $2+2s+4(1-b')>0$.
So we require $1/2<b\leq b'<(s+3)/2$.
This is possible if $s>-2$.
In the latter case, we reduce to
\begin{equation*}
\sum_{ N_2\ll N} \sum_{L_{max}\sim N^4N_2}\frac{N^{1+s}L^{1/2}_{min}N^{-2}}{N^s \langle N_2\rangle^s L^b_{min}L^b_{med} L^{1-b'}_{max}} \min\{H,L_{med}\}^{1/2} \lesssim 1.
\end{equation*}
Using $\min\{a,b\}\leq\sqrt{ab}$ and then performing the $L$ summations, we further reduce to
\begin{equation}\label{app11}
\sum_{ N_2\ll N} \frac{N^{1+s}N^{-2}}{N^s \langle N_2\rangle^s(N^4N_2)^{1-b'}}(N^4N_2)^{1/4} \lesssim 1.
\end{equation}
We then decompose the left-hand side of \eqref{app11} into two parts $N_2\leq 1$ and $N_2> 1$:
\begin{align*}
&\sum_{N_2\leq 1} \frac{N^{1+s}N^{-2}}{N^s \langle N_2\rangle^s(N^4N_2)^{1-b'}}(N^4N_2)^{1/4}
+ \sum_{1< N_2\leq N} \frac{N^{1+s}N^{-2}}{N^s \langle N_2\rangle^s(N^4N_2)^{1-b'}}(N^4N_2)^{1/4}\\
&=: K_1+K_2.
\end{align*}
It is easy to see that $K_1\lesssim N^{-4(1-b')}\lesssim 1$ if $b'<3/4$.
On the other hand,
$$K_2\lesssim\frac{\min\{1,N^{b'-s-3/4}\}}{N^{4(1-b')}}\lesssim1$$
if $b'\leq1$ when $b'-s-3/4\geq 0$, and if $4(1-b')-(b'-s-3/4)>0$ when $b'-s-3/4<0$.
So we get the desired bound for all $s$ with $1/2<b\leq b'<\min\{(4s+19)/20,s+3/4\}$.
This completes the proof.

\end{document}